\documentclass[12pt]{amsart}
\usepackage{kantlipsum}
\usepackage{chngcntr}
\calclayout
\usepackage{mathrsfs}
\usepackage{hyperref}
\usepackage{graphicx}
\usepackage{color}
\usepackage{caption}
\usepackage{amsthm}
\usepackage{amsmath}
\numberwithin{equation}{section}
\usepackage{amsfonts}

\counterwithin{table}{section}

\newtheorem{prop}{Proposition}[section]

\newtheorem{theorem}[prop]{Theorem}
\newtheorem{lemma}[prop]{Lemma}

\newtheorem{cor}[prop]{Corollary}
\newtheorem{conj}[prop]{Conjecture}
\theoremstyle{definition}

\newtheorem*{ex*}{Example}
\newtheorem*{rem*}{Remark}

\DeclareMathAlphabet{\mathpzc}{OT1}{pzc}{m}{it}
\DeclareMathOperator{\maxp}{maxp}

\DeclareMathOperator{\maxN}{maxN}

\title[Dyson's Partition Ranks]{Dyson's Partition Ranks and their Multiplicative Extensions}
\author[Elaine Hou]{Elaine Hou}
\address{Department of Mathematics, Yale University, New Haven, CT 06520}
\email{\href{mailto:elaine.hou@yale.edu}{{\tt elaine.hou@yale.edu}}}
\author[Meena Jagadeesan]{Meena Jagadeesan}
\address{Department of Mathematics, Harvard University, Cambridge, MA 02138}
\email{\href{mailto:mjagadeesan@college.harvard.edu}{{\tt mjagadeesan@college.harvard.edu}}}\thanks{The authors would like to thank the NSF and the Emory REU (especially Professor Ono) for their support of our research. }
\begin{document}
\maketitle
\begin{abstract}
We study the Dyson rank function $N(r,3;n)$, the number of partitions of $n$ with rank $\equiv r \pmod 3$. We investigate the convexity of these functions. We extend $N(r,3;n)$ multiplicatively to the set of partitions, and we determine the maximum value when taken over all partitions of size $n$. 
\end{abstract}
\section{Introduction}
For $n \in \mathbb{N}$, a \textit{partition} of a nonnegative integer $n$ is a finite sequence of nonincreasing natural numbers $\lambda := (\lambda_1, \lambda_2, \ldots, \lambda_k)$, where $\lambda_1 + \lambda_2 + \cdots + \lambda_k = n.$ As usual, let $p(n)$ denote the number of partitions of $n$. The study of partitions dates back to the eighteenth century, appearing in the work of Euler. Among the most famous properties of partitions are the following congruences, proved by Srinivasa Ramanujan in 1919 \cite{ram}: 
\begin{align*}
p(5n+4) &\equiv 0 \pmod 5 \\
p(7n+5) &\equiv 0 \pmod 7 \\
p(11n+6) &\equiv 0 \pmod {11}. 
\end{align*}
In the 1940s, Freeman Dyson aimed to find a combinatorial explanation for these congruences. He sought a combinatorial statistic that divides the partitions of $5n+4$ (resp. $7n+5$, $11n+6$) into $5$ (resp. $7$, $11$) groups of equal size. He found the rank statistic \cite{dyson}.

The \textit{rank} of a partition $\lambda := (\lambda_1, \lambda_2, \ldots, \lambda_k)$ is $\lambda_1 - k$: that is, the size of the largest part minus the number of parts. Let $N(m,n)$ be the number of partitions of $n$ with Dyson rank $m$. The generating function of $N(m,n)$ is the following:
\begin{equation}
\label{gf}
R(w; q) := 1 + \sum_{n=1}^{\infty} \sum_{m=-\infty}^{\infty} N(m, n) w^m q^n = 1 + \sum_{n=1}^{\infty} \frac{q^{n^2}}{(wq;q)_n(w^{-1}q;q)_n}, 
\end{equation}
where $(a;q)_n := (1-a)(1-aq)\cdots(1-aq^{n-1}).$
Understanding Ramanujan's congruences using Dyson's rank requires the following variant. Let $N(r,t;n)$ be the number of partitions of $n$ with rank $\equiv r \pmod t$. Using this notion, Dyson conjectured (and Atkin and Swinnerton-Dyer later proved \cite{swin}) that for each $m$, we have
\begin{align*}
N(m,5;5n+4) &= \frac{1}{5}p(5n+4) \\
N(m,7;7n+5) &= \frac{1}{7}p(7n+5). \\
\end{align*}
This confirms that the rank statistic provides a combinatorial proof\footnote{Dyson's rank does not explain Ramanujan's congruence modulo $11$.}
 of Ramanujan's congruences modulo $5$ and $7$. 
Ramanujan, in a joint work with Hardy, also proved the following asymptotic formula \cite{apostol}:
\begin{equation}
\label{asym} 
p(n) \sim \frac{1}{4n\sqrt{3}} e^{\pi \sqrt{\frac{2n}{3}}}. 
\end{equation}
Using a refinement of this asymptotic due to Lehmer \cite{lehmer}, Bessenrodt and Ono \cite{bess} recently proved that the partition function satisfies the following convexity property. If $a,b$ are integers with $a,b >1$ and $a+b>9$, then
\[p(a)p(b) > p(a+b). \]
In view of this, it is natural to ask whether Dyson's rank functions $N(r, t;n)$ also satisfy convexity. We prove this for each $r =0, 1, 2$ and $t=3$ for all but a finite number of $a$ and $b$.  
\begin{theorem}
\label{ineq}
If $r=0$ (resp. $r=1,2$), then 
\[N(r, 3; a)N(r,3;b) > N(r,3;a+b). \]
for all $a,b \ge 12$ (resp. $11$, $11$).
\end{theorem}
\begin{rem*}
Notice that this bound is sharp for $r=0$ (resp. $1$, $2$) for $a, b = 11$ (resp. $10$, $10$). Namely, we have that
\begin{align*}
N(0, 3; 11)N(0,3;11) = 16 \cdot 16 &< 340 = N(0,3;22) \\
N(1, 3;10)N(1,3;10) = 13 \cdot 13 &< 211 = N(1,3;20) \\
N(2, 3; 10)N(2,3;10) = 13 \cdot 13 &< 211 = N(2,3;20).
\end{align*}
\end{rem*}
Bessenrodt and Ono \cite{bess} used their convexity result to study the multiplicative extension of the partition function defined by  
\begin{equation}
p(\lambda) := \prod_{j = 1}^k p(\lambda_j),
\end{equation}
where $\lambda = (\lambda_1, \lambda_2, \ldots, \lambda_k)$ is a partition.
For example, if $\lambda = (5, 3, 2)$, then $p(\lambda) = p(5)p(3)p(2) = 42.$  They then studied the maximum of this function on $P(n)$, the set of all partitions of $n$. The maximal value is defined as \[\maxp(n) := \max(p(\lambda) : \lambda \in P(n)). \] Their main result was a closed formula for $\maxp(n)$, and they also fully characterized all partitions $\lambda \in P(n)$ that achieve this maximum.
We carry out a similar analysis for the functions $N(r,t;n)$ in the case of $t=3$. We extend each $N(r,3;n)$ to partitions by
\begin{equation}
N(r, 3; \lambda) := \prod_{j = 1}^k N(r, 3; \lambda_j).  
\end{equation}
We determine the maximum of each function on $P(n)$, where the maximal value is defined as
\begin{equation}
\maxN(r, 3; n) := \max(N(r, 3;\lambda) : \lambda \in P(n)). 
\end{equation}
We also fully characterize all $\lambda \in P(n)$ that achieve each maximum.
\begin{theorem}
\label{bash3}
Assume the notation above. Then the following are true:
\begin{enumerate}
\item{If $n \ge 33$, then we have that
\[ 
\maxN(0,3;n) = 
\begin{cases}
7^{\frac{n}{7}} & n \equiv 0 \pmod 7 \\
37^2 \cdot 16 \cdot 7^{\frac{n-36}{7}} & n \equiv 1 \pmod 7 \\
37 \cdot 16 \cdot 7^{\frac{n-23}{7}} & n \equiv 2 \pmod 7 \\
16 \cdot 7^{\frac{n-10}{7}} & n \equiv 3 \pmod 7 \\
37^3 \cdot 7^{\frac{n-39}{7}} & n \equiv 4 \pmod 7 \\
37^2 \cdot 7^{\frac{n-26}{7}} & n \equiv 5 \pmod 7 \\
37 \cdot 7^{\frac{n-13}{7}} & n \equiv 6 \pmod 7, \\
\end{cases} \]
and it is achieved at the unique partitions
\begin{align*}
(7, 7, \ldots, 7) &\textit{ when } n \equiv 0 \pmod 7 \\
(13, 13, 10,  7, \ldots, 7) &\textit{ when } n \equiv 1 \pmod 7  \\
(13, 10, 7, \ldots, 7) &\textit{ when } n \equiv 2 \pmod 7  \\
(10, 7, \ldots, 7) &\textit{ when } n \equiv 3 \pmod 7 \\
(13, 13, 13, 7, \ldots, 7) &\textit{ when } n \equiv 4 \pmod 7  \\
(13, 13, 7, \ldots, 7) &\textit{ when } n \equiv 5 \pmod 7  \\
(13, 7, \ldots, 7) &\textit{ when } n \equiv 6 \pmod 7. \\
\end{align*}
}
\item{If $n \ge 22$, then we have that
\[ 
\maxN(1,3;n) = \maxN(2,3;n) = 
\begin{cases}
46^{\frac{n}{14}} &\textit{ when } n \equiv 0 \pmod {14} \\
59 \cdot 46^{\frac{n-15}{14}} &\textit{ when } n \equiv 1 \pmod {14} \\
59^2 \cdot 46^{\frac{n-30}{14}} &\textit{ when } n \equiv 2 \pmod {14} \\
101 \cdot 46^{\frac{n-17}{14}} &\textit{ when } n \equiv 3 \pmod {14} \\
101 \cdot 59 \cdot 46^{\frac{n-32}{14}} &\textit{ when } n \equiv 4 \pmod {14}\\
20^3 \cdot 46^{\frac{n-33}{14}} &\textit{ when } n \equiv 5 \pmod {14} \\
26 \cdot 20^2 \cdot 46^{\frac{n-34}{14}} &\textit{ when } n \equiv 6 \pmod {14}\\
26^2 \cdot 20 \cdot 46^{\frac{n-35}{14}} &\textit{ when } n \equiv 7 \pmod {14}\\
20^2 \cdot 46^{\frac{n-22}{14}} &\textit{ when } n \equiv 8 \pmod {14} \\
26 \cdot 20 \cdot 46^{\frac{n-23}{14}} &\textit{ when } n \equiv 9 \pmod {14} \\
26^2 \cdot 46^{\frac{n-24}{14}} &\textit{ when } n \equiv 10 \pmod {14} \\
20 \cdot 46^{\frac{n-11}{14}} &\textit{ when } n \equiv 11 \pmod {14}\\
26 \cdot 46^{\frac{n-12}{14}} &\textit{ when } n \equiv 12 \pmod {14} \\
59 \cdot 26 \cdot 46^{\frac{n-27}{14}} &\textit{ when } n \equiv 13 \pmod {14}, \\
\end{cases} \]
and it is achieved at the unique partitions
\begin{align*}
(14, 14, \ldots, 14) &\textit{ when } n \equiv 0 \pmod {14} \\
(15, 14, \ldots, 14) &\textit{ when } n \equiv 1 \pmod {14} \\
(15, 15, 14, \ldots, 14) &\textit{ when } n \equiv 2 \pmod {14} \\
(17, 14, \ldots, 14) &\textit{ when } n \equiv 3 \pmod {14} \\
(17, 15, 14, \ldots, 14) &\textit{ when } n \equiv 4 \pmod {14} \\
(11, 11, 11, 14, \ldots, 14) &\textit{ when } n \equiv 5 \pmod {14} \\
(12, 11, 11, 14, \ldots, 14) &\textit{ when } n \equiv 6 \pmod {14} \\
(12, 12, 11, 14, \ldots, 14) &\textit{ when } n \equiv 7 \pmod {14} \\
(11, 11, 14, \ldots, 14) &\textit{ when } n \equiv 8 \pmod {14} \\
(12, 11, 14, \ldots, 14) &\textit{ when } n \equiv 9 \pmod {14} \\
(12, 12, 14, \ldots, 14) &\textit{ when } n \equiv 10 \pmod {14}\\
(11, 14, \ldots, 14) &\textit{ when } n \equiv 11 \pmod {14}\\
(12, 14, \ldots, 14) &\textit{ when } n \equiv 12 \pmod {14}\\
(15, 12, 14, \ldots, 14) &\textit{ when } n \equiv 13 \pmod {14}. 
\end{align*}
}
\end{enumerate}
\end{theorem}
In Section 2, we prove Theorem~$\ref{ineq}$ by finding explicit upper and lower bounds for $N(r,3;n)$ using the work of Lehmer \cite{lehmer} and Bringmann \cite{bring}. In Section 3, we prove Theorem~$\ref{bash3}$ by applying the convexity property together with combinatorial arguments. In Section 4, we discuss potential extensions of our results to other values of $t$.
\section{Proof of Theorem~$\ref{ineq}$}
Theorem~$\ref{ineq}$ states that
\begin{equation} 
\label{badthm}
N(r, 3; a)N(r,3;b) > N(r,3;a+b)
\end{equation} for $r=0$ (resp. $1, 2$) for $a, b \ge 12$ (resp. $11, 11$). Essentially, this implies the convexity of Dyson's rank functions $N(r,3;n)$. We prove $(\ref{badthm})$ for $a, b \ge 500$ by finding a lower bound for $N(r,3;a)N(r,3;b)$ and an upper bound for $N(r,3;a+b)$. We verify the remaining cases using a computer program.
\subsection{Preliminaries for the Proof of Theorem~$\ref{ineq}$}
In order to obtain bounds for $N(r,3;n)$, we use methods in analytic number theory. We use analytic estimates due to Lehmer and Bringmann in order to study and bound $N(r,3;n)$.

For $p(n)$, we use the explicit bounds provided by Lehmer \cite{bess}.
\begin{theorem}[Lehmer]
If n is a positive integer and $\mu=\mu(n):=\frac{\pi}{6}\sqrt{24n-1}$, then
\begin{equation}
p(n) = \frac{\sqrt{12}}{24n-1}\left[\left(1-\frac{1}{\mu}\right)e^{\mu}+ \left(1+\frac{1}{\mu}\right)e^{-\mu} \right]+E(n), 
\end{equation}
where we have that
\begin{equation}
\left|E(n)\right| < \frac{\pi^{2}}{\sqrt{3}}\left[\frac{1}{\mu^{3}}\sinh(\mu) + \frac{1}{6} - \frac{1}{\mu^{2}}\right].
\end{equation}
\end{theorem}
Now, by using the Lehmer bound, Bessenrodt and Ono \cite{bess} obtained bounds $p^{L}(n)$ and $p^{U}(n)$ on $p(n)$ that satisfy 

\begin{equation}
\label{bessbd}
p^{L}(n) < p(n) < p^{U}(n),   
\end{equation}
where $p^{L}(n)$ and $p^{U}(n)$ are defined as follows:
\begin{align*}
p^{L}(n) &:= \frac{\sqrt3}{12n}\left(1-\frac{1}{\sqrt{n}}\right)e^{\mu},\\
p^{U}(n) &:= \frac{\sqrt3}{12n}\left(1+\frac{1}{\sqrt{n}}\right)e^{\mu}.
\end{align*}

Bringmann estimates and obtains asymptotics for $R(\zeta_c^a;q)$ for positive integers $a < c$ in the case that $c$ is odd and where $\zeta_c^a := e^{\frac{2\pi i a}{c}}$. She uses the following notation: 
\begin{equation}
R(\zeta_c^a; q) =: 1 + \sum_{n=1}^{\infty} A\left(\frac{a}{c}; n\right) q^{n}. 
\end{equation}

Here, we recall a precise version of her result in the special case that $a=1$ and $c=3$. We use the following notation:
\begin{align*}
\widetilde{E}_1(n) &:= \frac{12}{(24n-1)^{1/2}} \sum_{k=2}^{\frac{\sqrt{n}}{3}} k^{\frac{1}{2}} \cdot \sinh{\left(\frac{\pi}{18k} \sqrt{24n-1} \right)}, \\
\widetilde{E}_2(n) &:= \frac{0.12 \cdot e^{2\pi + \frac{\pi}{24}}}{\sqrt{3}} \sum_{k=1}^{\frac{\sqrt{n}}{3}} k^{-\frac{1}{2}}, \\
\widetilde{E}_3(n) &:= 1.412 \sqrt{3} \cdot e^{2\pi} \sum_{1 \le k \le \sqrt{n}, 3 \not| k} k^{-\frac{1}{2}}, \\
\widetilde{E}_4(n) &:= 2 \sqrt{3} e^{2\pi + \frac{\pi}{12}} \cdot n^{-1/2} \sum_{1 \le k \le \frac{\sqrt{n}}{3}} k^{\frac{1}{2}}, \\
\widetilde{E}_5(n) &:= 8\pi \cdot e^{2\pi + \frac{\pi}{24}} \cdot n^{-3/4} \sum_{1 \le k \le \frac{\sqrt{n}}{3}} k, \\
\widetilde{E}_6(n) &:=  2^{\frac{1}{4}} \cdot (e + e^{-1}) \cdot e^{2\pi} \cdot n^{-1/4} \sum_{1 \le k \le \sqrt{n}} \frac{1}{k} \sum_{v=1}^{k} \left(\min\left(\left\{\frac{v}{k} - \frac{1}{6k} + \frac{1}{3} \right\}, \left\{\frac{v}{k} - \frac{1}{6k} - \frac{1}{3} \right\} \right) \right)^{-1}.
\end{align*}

Bringmann proves the following result regarding the main term $M(n)$ and bound on the error term $E_R$ (see p. 18-19 of \cite{bring}):

\begin{theorem}[Bringmann]
\label{bringbound}
For $n \in \mathbb{N}$, let $M(n)$ be
\[M(n) := -\frac{8\sin{\left(\frac{\pi}{18} - \frac{2n\pi}{3} \right)} \sinh{\left( \frac{\pi}{18} \sqrt{24n-1} \right)}}{\sqrt{24n-1}}. \]
Then we have that 
\[ A\left(\frac{1}{3};n\right) = M(n) + E_R(n), \]
where
\[ E_R(n)  := \sum_{i=1}^{6} E_i(n), \]
and each $E_i(n)$ is bounded as follows:  
\[|E_i(n)| \le \widetilde{E}_i(n).
 \]
\end{theorem}

\begin{rem*}
One can find explicit definitions of $E_1(n), \ldots, E_6(n)$ scattered throughout \cite{bring}. 
\end{rem*}


\subsection{Explicit Bounds for Error Terms}


In order to prove Theorem~$\ref{ineq}$, we must effectively bound each of the error terms $\widetilde{E}_1(n), \ldots, \widetilde{E}_6(n)$. First, we obtain $L(n)$, a lower bound for $M(n)$, and $U(n)$, an upper bound for $M(n)$ by using the fact that for any integer $n$, $|\sin{\frac{\pi}{18}}| \le |\sin{(\frac{\pi}{18} - \frac{2n\pi}{3})}| \le 1$. Thus, we have that the following is true:

\[ L(n) := \left|\frac{8\sin{\left(\frac{\pi}{18} \right)} \sinh{\left( \frac{\pi}{18} \sqrt{24n-1} \right)}}{\sqrt{24n-1}}\right| \le |M(n)| \le  \left|\frac{8 \sinh{\left( \frac{\pi}{18} \sqrt{24n-1} \right)}}{\sqrt{24n-1}} \right| =: U(n). \]

In Subsections 2.2.1-2.2.6, we prove the following bounds for each $\widetilde{E}_i(n)$:

\begin{prop}
\label{c_i}
For $i = 1, 2, \ldots, 6$ and for $n \ge 500$, we have that 
\[\frac{\widetilde{E}_i(n)}{L(n)} < c_i,\] 
where all $c_i$ are listed in Table$~\ref{c_itable}$. 
\begin{table}
\caption{}
\label{c_itable}
\begin{center}
\begin{tabular}{|| c c ||}
\hline
$i$ & $c_i$ \\
\hline
$1$ & $0.0065$ \\
\hline
$2$ & $0.00019$ \\
\hline
$3$ & $0.0098$ \\
\hline
$4$ & $0.0071$ \\
\hline
$5$ & $0.0072$ \\
\hline
$6$ & $0.54$ \\
\hline
\end{tabular} 
\end{center}
\end{table}
\end{prop}
Using Proposition$~\ref{c_i}$, we obtain the following bound for $E_R(n)$:
\begin{cor}
\label{boundingours}
Assume the notation above. Then for $n \ge 500$, the following is true: \[|E_R(n)| \le 0.58L(n).\]
\end{cor}
\begin{proof}
This follows from adding the bounds for each $\widetilde{E}_i(n)$ in Proposition~$\ref{c_i}$ and applying Theorem~$\ref{bringbound}$.
\end{proof}


\subsubsection{Effective Bounds for Error $\widetilde{E}_1(n)$}
We prove Proposition~$\ref{c_i}$ for $i = 1$. We approximate the finite sum in $\widetilde{E}_1(n)$ by the number of terms multiplied by the summand evaluated at $k=2$ (since this is the largest summand). For $n \ge 500$, we then have that 
\[\widetilde{E}_1(n) \le \frac{\sqrt{n}}{3} \left( \frac{12}{(24n-1)^{1/2}} 2^{\frac{1}{2}} \cdot \sinh{\left(\frac{\pi}{36} \sqrt{24n-1} \right)} \right).\]
Now, we consider the ratio of our bound of $\widetilde{E}_1(n)$ to $L(n)$. We have that
\[\frac{\widetilde{E}_1(n)}{L(n)} \le \frac{\frac{\sqrt{n}}{3} \left( \frac{12}{(24n-1)^{1/2}} 2^{\frac{1}{2}} \cdot \sinh{\left(\frac{\pi}{36} \sqrt{24n-1} \right)} \right)}{L(n)} = \frac{\sqrt{n}  \sinh{\left(\frac{\pi}{36} \sqrt{24n-1} \right)}}{\sqrt{2} \sin{\left(\frac{\pi}{18} \right)} \sinh{\left(\frac{\pi}{18} \sqrt{24n-1} \right)}} =: F_1(n). \]
It is easy to check that $F_1(n)$ is a decreasing function of $n$ for $n \ge 500$. This means 
that 
\[\widetilde{E}_1(n) \le F_1(500) L(n) \le 0.0065 L(n).\]
\subsubsection{Effective Bounds for Error $\widetilde{E}_2(n)$}
We prove Proposition~$\ref{c_i}$ for $i = 2$. We find an upper bound for $\widetilde{E}_2(n)$ by using that, for $n \ge 500$, 
\[\sum_{k=1}^{\frac{\sqrt{n}}{3}} k^{-\frac12} \le \int_0^{\frac{\sqrt{n}}{3}} k^{-\frac12}dk.\] 
For $n \ge 500$, we have that 
\[\widetilde{E}_2(n) \le \frac{0.12 \cdot e^{2\pi + \frac{\pi}{24}}}{\sqrt{3}} \int_{0}^{\frac{\sqrt{n}}{3}} k^{-\frac{1}{2}}dk \\
\le 0.08 \cdot e^{2\pi + \frac{\pi}{24}} n^{\frac{1}{4}}. \]
Now, we consider the ratio of our bound of $\widetilde{E}_2(n)$ to $L(n)$. We have that
\[\frac{\widetilde{E}_2(n)}{L(n)} \le \frac{0.08 \cdot e^{2\pi + \frac{\pi}{24}} n^{\frac{1}{4}}}{L(n)} = \frac{0.01 e^{2\pi + \frac{\pi}{24}} n^{\frac{1}{4}} \sqrt{24n - 1}}{\sin{\left(\frac{\pi}{18} \right)} \sinh{\left(\frac{\pi}{18} \sqrt{24n-1} \right)}} =: F_2(n). \]
It is easy to check that $F_2(n)$ is a decreasing function of $n$ for $n \ge 500$. This means 
that 
\[\widetilde{E}_2(n) \le F_2(500) L(n) \le 0.0019 L(n).\]
\subsubsection{Effective Bounds for Error $\widetilde{E}_3(n)$}
We prove Proposition~$\ref{c_i}$ for $i = 3$. We estimate $\widetilde{E}_3(n)$ by using our method from Subsection 2.2.2. For $n \ge 500$, we have that
\[\widetilde{E}_3(n) \le 1.412 \sqrt{3} \cdot e^{2\pi} \int_{0}^{\sqrt{n}} k^{-\frac{1}{2}}dk \le 2.824 \sqrt{3} \cdot e^{2\pi} n^{\frac{1}{4}}. \]
Now, we consider the ratio of our bound of $\widetilde{E}_3(n)$ to $L(n)$. We have that
\[\frac{\widetilde{E}_3(n)}{L(n)} \le \frac{2.824 \sqrt{3} \cdot e^{2\pi} n^{\frac{1}{4}}}{L(n)} \le \frac{2.824 \sqrt{3} \cdot e^{2\pi} n^{\frac{1}{4}} \sqrt{24n-1}}{8 \sin{\left(\frac{\pi}{18} \right)} \sinh{\left(\frac{\pi}{18} \sqrt{24n-1} \right)}} =: F_3(n). \]
It is easy to check that $F_3(n)$ is a decreasing function of $n$ for $n \ge 500$. This means 
that 
\[\widetilde{E}_3(n) \le F_3(500) L(n) \le 0.0098 L(n).\]

\subsubsection{Effective Bounds for Error $\widetilde{E}_4(n)$}
We prove Proposition~$\ref{c_i}$ for $i = 4$. We find an upper bound for $\widetilde{E}_4(n)$ by using that, for $n \ge 500$, we have that \[\sum_{k=1}^{\frac{\sqrt{n}}{3}} k^{\frac12} \le \int_0^{\frac{\sqrt{n}}{2}} k^{\frac12}dk.\]

For $n \ge 500$, we have that
\[\widetilde{E}_4(n) \le 2 \sqrt{3} e^{2\pi + \frac{\pi}{12}} \cdot n^{-1/2} \int_{0}^{\frac{\sqrt{n}}{2}} k^{\frac{1}{2}}dk \le \frac{\sqrt{6}}{3} e^{2\pi + \frac{\pi}{12}} \cdot n^{\frac{1}{4}}.\]

Now, we consider the ratio of our bound of $\widetilde{E}_4(n)$ to $L(n)$. We have that
\[\frac{\widetilde{E}_4(n)}{L(n)} \le \frac{\frac{\sqrt{6}}{3} e^{2\pi + \frac{\pi}{12}} \cdot n^{\frac{1}{4}}}{L(n)} \le \frac{\sqrt{6} e^{2\pi + \frac{\pi}{12}} \cdot n^{\frac{1}{4}}\sqrt{24n-1}}{24 \sin{\left(\frac{\pi}{18} \right)} \sinh{\left(\frac{\pi}{18} \sqrt{24n-1} \right)}} =: F_4(n). \]
It is easy to check that $F_4(n)$ is a decreasing function of $n$ for $n \ge 500$. This means 
that 
\[\widetilde{E}_4(n) \le F_4(500) L(n) \le 0.0071 L(n).\]

\subsubsection{Effective Bounds for Error $\widetilde{E}_5(n)$}
We prove Proposition~$\ref{c_i}$ for $i = 5$. We find an upper bound for $\widetilde{E}_5(n)$ by using the methods from Subsection 2.2.4. For $n \ge 500$, we have that
\[\widetilde{E}_5(n) \le 8\pi \cdot e^{2\pi + \frac{\pi}{24}} \cdot n^{-3/4} \int_{0}^{\frac{\sqrt{n}}{2}} kdk \le \pi \cdot e^{2\pi + \frac{\pi}{24}} \cdot n^{\frac{1}{4}}. \]
Now, we consider the ratio of our bound of $\widetilde{E}_5(n)$ to $L(n)$. We have that
\[\frac{\widetilde{E}_5(n)}{L(n)} \le \frac{\pi \cdot e^{2\pi + \frac{\pi}{24}} \cdot n^{\frac{1}{4}}}{L(n)} \le \frac{\pi \cdot e^{2\pi + \frac{\pi}{24}} \cdot n^{\frac{1}{4}} \sqrt{24n-1}}{8 \sin{\left(\frac{\pi}{18} \right)} \sinh{\left(\frac{\pi}{18} \sqrt{24n-1} \right)}} =: F_5(n). \]
It is easy to check that $F_5(n)$ is a decreasing function of $n$ for $n \ge 500$. This means 
that 
\[\widetilde{E}_5(n) \le F_5(500) L(n) \le 0.0072 L(n).\]
\subsubsection{Effective Bounds for Error $\widetilde{E}_6(n)$}
We prove Proposition~$\ref{c_i}$ for $i = 6$. First, we notice that  $\left(\min\left(\left\{\frac{v}{k} - \frac{1}{6k} + \frac{1}{3} \right\}, \left\{\frac{v}{k} - \frac{1}{6k} - \frac{1}{3} \right\} \right) \right)^{-1} \le 6k$. We estimate the sum with the methods from Subsections 2.2.4 and 2.2.5. For $n \ge 500$, we have that
\begin{align*}
\widetilde{E}_6(n) &\le  2^{\frac{1}{4}} \cdot (e + e^{-1}) \cdot e^{2\pi} \cdot n^{-1/4} \int_{1}^{\sqrt{n}+1} 6k dk \\
&\le 2^{\frac{1}{4}} \cdot (e + e^{-1}) \cdot e^{2\pi} \cdot 3(n^{\frac{3}{4}} + 2n^{\frac14}).
\end{align*}
Now, we consider the ratio of our bound of $\widetilde{E}_6(n)$ to $L(n)$. We have that
\[\frac{\widetilde{E}_6(n)}{L(n)} \le \frac{2^{\frac{1}{4}} \cdot (e + e^{-1}) \cdot e^{2\pi} \cdot 3(n^{\frac{3}{4}} + 2n^{\frac14})}{L(n)} \le \frac{2^{\frac{1}{4}} \cdot (e + e^{-1}) \cdot e^{2\pi} \cdot 3(n^{\frac{3}{4}} + 2n^{\frac14}) \sqrt{24n-1}}{8 \sin{\left(\frac{\pi}{18} \right)} \sinh{\left(\frac{\pi}{18} \sqrt{24n-1} \right)}} =: F_6(n). \]
It is easy to check that $F_6(n)$ is a decreasing function of $n$ for $n \ge 500$. This implies 
that 
\[\widetilde{E}_6(n) \le F_6(500) L(n) \le 0.54 L(n).\] 

\subsection{Proof of Theorem~$\ref{ineq}$}
In order to prove convexity for $N(r,3;n)$, we first write $N(r,3;n)$ in terms of $p(n)$ and $A(\frac{1}{3}; n)$. We have the following generating function of $N(r,t;n)$ for all $t$, where we use the special case that $t=3$:
\begin{prop}
\label{ngf}
For nonnegative integers $r$ and $t$, we have that
\[
1 + \sum_{n=1}^{\infty} N(r, t; n) q^n  = 1+ \frac{1}{t} \left[\sum_{n=1}^{\infty} p(n)q^n + \sum_{j=1}^{t-1} \zeta_t^{-rj} R(\zeta_t^{j};q) \right],\]
where $\zeta_t := e^{2\pi i/t}.$
\end{prop}
\begin{proof}
For $r$ and $t$ as defined above, we have that
\begin{align*}
1+ \frac{1}{t} \left[\sum_{j=0}^{t-1} \zeta_t^{-rj} R(\zeta_t^{j};q) \right] 
&=1+ \frac{1}{t}\sum_{j=0}^{t-1} \sum_{n=0}^{\infty} \sum_{m=-\infty}^{\infty} N(m,n) \zeta_t^{-rj}\zeta_t^{mj}q^n \\
&= 1+ \frac{1}{t}\sum_{n=0}^{\infty} \sum_{m=-\infty}^{\infty} \sum_{j=0}^{t-1} N(m,n)\zeta_t^{(m-r)j}q^n.
\end{align*}
Notice that for each $m \not\equiv r \pmod t$, because the sum is over the complete set of $t$th roots of unity, the coefficient of $q^n$ vanishes. For each $m \equiv r \pmod t$, the coefficient of $q^n$ is equal to $N(m,n)$. Hence, we obtain the desired result. 
\end{proof}

We can now determine an explicit bound for $N(r,3;n)$:
\begin{prop}
\label{boundforn}
For $r$ defined as above and $n \ge 500$, we have the following bound for $N(r,3;n)$: 
\[p^L(n) - 2\left|A\left(\frac{1}{3};n\right)\right| \le |N(r;3, n)| \le p^U(n) + 2\left|A\left(\frac{1}{3};n\right)\right|.\]
\end{prop}
\begin{proof}
First, we have that 
$A\left(\frac{1}{3}; n\right) = A\left(\frac{2}{3}; n\right)$ by the symmetry of the third roots of unity.
This fact, together with Proposition~$\ref{ngf}$, yields the following:
\[p(n) - 2\left|A\left(\frac{1}{3};n\right)\right| \le |N(r;3, n)| \le p(n) + 2\left|A\left(\frac{1}{3};n\right)\right|.\]
Now, we apply ($\ref{bessbd}$) to obtain the desired result.
\end{proof}
We use Corollary~$\ref{boundingours}$ together with Proposition~$\ref{boundforn}$ to obtain the following upper and lower bounds for $N(r,3;n)$: 
\begin{prop}
\label{almosthere}
Assume the notation above. Then for $n\ge 500$, the following is true:
\[\frac{1}{3}(1-0.01)p^L(n) < N(r,t;n) < \frac{1}{3}(1+0.01)p^U(n).\]
\end{prop}
\begin{proof}
By Corollary~$\ref{boundingours}$, we have that  
\[\left|A\left(\frac{j}{3};n\right)\right| \le 1.58U(n).\] 
It is easy to check that
\[\frac{U(n)}{p^L(n)} = \frac{2\sqrt{3} \sinh{\left(\frac{\pi}{18} \sqrt{24n-1}\right)}}{3n \sqrt{24n-1} e^{\frac{\pi \sqrt{24n-1}}{6}}\left(1 - \frac{1}{\sqrt{n}} \right)} \]
is a decreasing function in $n$ for $n \ge 500$. 
As a result, we obtain the following:
\[\left|A\left(\frac{j}{3};n\right)\right| \le \frac{1.58U(500)}{p^L(500)} p^L(n) < 0.005 p^L(n). \]
We apply this to Proposition~$\ref{boundforn}$ to obtain the desired result.
\end{proof}
We now use Proposition~$\ref{almosthere}$ together with an argument similar to that of Bessenrodt and Ono (see p. 2-3 of \cite{bess}) to prove Theorem~$\ref{ineq}$ for $a, b \ge 500$. We first define the following notation: 
\[S_x(\lambda) := \frac{(1+ \frac{1}{\sqrt{x + \lambda x}})}{(1- \frac{1}{\sqrt{x}})(1- \frac{1}{\sqrt{\lambda x}})} \]
\[T_x(\lambda) := \frac{\pi}{6} \left(\sqrt{24x-1} + \sqrt{24\lambda x-1} - \sqrt{24(x + \lambda x)-1} \right). \]
We use the following Lemma in our proof:
\begin{lemma}
Assume the notation above. Suppose that for a fixed $0 < c < 1$ and for any nonnegative integers $t$ and $n$, we have that
\[\frac{1}{t}(1-c)p^L(n) < N(r,t;n) < \frac{1}{t}(1+c)p^U(n).\]
Then we have that
\[ N(r,t;a)N(r,t;b) > N(r,t;a+b)\]
for all $a,b \ge x$, where $x$ is the minimum value satisfying
\[T_x(1) > \log\left(4x\sqrt{3}t\frac{1+c}{(1-c)^2}\right) + \log(S_x(1)). \]
\end{lemma}
\begin{proof}
We may assume that $1 < a \le b$ for convenience, so we will let $b = \lambda a$. These inequalities give us
\begin{align*}
N(r,t;a)N(r,t;\lambda a) &> \frac{1}{48\lambda a^2}\cdot \frac{1}{t^3}(1-c)^2 \left(1- \frac{1}{\sqrt{a}}\right)\left(1- \frac{1}{\sqrt{\lambda a}}\right) e^{\mu(a) + \mu(\lambda a)}\\
N(r,t;a+\lambda a) &< \frac{\sqrt{3}}{12(a + \lambda a)} \left(1 + \frac{1}{\sqrt{a + \lambda a}}\right) e^{\mu(a + \lambda a)} \frac{1}{t}(1+c).
\end{align*} 
For all but finitely many cases, it suffices to find conditions on $a >1$ and $\lambda \ge 1$ for which
\[\frac{1}{48\lambda a^2}\cdot \frac{1}{t^2}(1-c)^2 \left(1- \frac{1}{\sqrt{a}}\right)\left(1- \frac{1}{\sqrt{\lambda a}}\right) e^{\mu(a) + \mu(\lambda a)} 
> \frac{\sqrt{3}}{12(a + \lambda a)} \left(1 + \frac{1}{\sqrt{a + \lambda a}}\right) e^{\mu(a + \lambda a)} \frac{1}{t}(1+c).\]
We have that $\frac{\lambda}{\lambda +1} \le 1$, so it suffices to consider when
\[e^{\mu(a) + \mu(\lambda a) - \mu(a + \lambda a)} > 4a\sqrt{3}t\frac{1+c}{(1-c)^2} S_a(\lambda).\]
By taking the natural log, we obtain the inequality 
\[T_a(\lambda) > \log\left(4a\sqrt{3}t\frac{1+c}{(1-c)^2}\right) + \log(S_a(\lambda)). \]
Simple calculations reveal that $S_a(\lambda)$ is decreasing for $\lambda \ge 1$, while $T_a(\lambda)$ is increasing in $\lambda \ge 1$. Therefore, we consider
\[T_a(\lambda) \ge T_a(1) > \log\left(4a\sqrt{3}t\frac{1+c}{(1-c)^2}\right) + \log(S_a(1)) \ge \log\left(4a\sqrt{3}t\frac{1+c}{(1-c)^2}\right) + \log(S_a(\lambda)). \]
\end{proof}
It can be verified that
\[T_x(1) > \log\left(12x\sqrt{3}\frac{1+0.01}{(1-0.01)^2}\right) + \log(S_x(1)) \]
for $x \ge 500$. This means that
\begin{equation}
\label{three}
N(r,3;a)N(r,3;b) > N(r, 3; a+b)
\end{equation}
for $a, b \ge 500$.  
We used Sage to confirm $(\ref{three})$ for $500 \ge \max(a,b) \ge 12$ (resp. $11$, $11$) for $r=0$ (resp. $r=1,2$). This proves Theorem~$\ref{ineq}$. 

\section{Proof of Theorem~$\ref{bash3}$}
In this section, we prove Theorem~$\ref{bash3}$. We compute the maximum of the multiplicative extension $N(r,3;\lambda)$ over all partitions of $n$. In addition, we identify the partitions that attain these values. These results are deduced from Theorem~$\ref{ineq}$. Since there are $3$ residue classes modulo $3$ and we have that\footnote{This follows immediately from considering conjugations of Ferrers diagrams. } $N(1, 3; n) = N(2, 3; n)$, we split our computation into two cases: $r=0$ and $r=1,2$. In Section 3.1, we compute $\maxN(0,3;n)$. In Section 3.2, we compute $\maxN(1,3;n)$.   
\subsection{Proof of Theorem~$\ref{bash3}$ for $r=0$}
In Subsection 3.1.1, we prove some combinatorial properties of $N(0, 3;\lambda)$ resulting from Theorem~$\ref{ineq}$ and the values of $N(0,3;n)$ for small $n$. In Subsection 3.1.2, we use these properties to deduce Theorem~$\ref{bash3}$ for $r=0$.
\subsubsection{Some combinatorics for $r=0$}
We require the values of $N(0,3;n)$ for $n\leq32$. These values are given in the first two columns of Table~$\ref{03}$, which were computed using Sage. We prove the correctness of the values in the last two columns over the course of this section. 
\begin{table}
\caption{}
\label{03}
\begin{center}
\begin{tabular}{|| c c c c ||}
\hline
$n$ & $N(0, 3;n)$ & $\maxN(0,3;n)$ & $\lambda$ \\
\hline
\hline
1 & 1 & 1 & (1) \\
\hline
2 & 0 & 1 & (1, 1) \\
\hline
3 & 1 & 1 & (3), (1, 1, 1) \\
\hline
4 & 3 & 3 & (4) \\
\hline
5 & 1 & 3 & (4, 1) \\
\hline
6 & 3 & 3 & (6), (4, 1, 1) \\
\hline
7 & 7 & 7 & (7) \\
\hline
8 & 6 & 9 & (4, 4) \\
\hline
9 & 10 & 10 & (9) \\
\hline
10 & 16 & 16 & (10) \\
\hline
11 & 16 & 21 & (7, 4) \\
\hline
12 & 25 & 27 & (4, 4, 4) \\
\hline
13 & 37 & 37 & (13) \\
\hline
14 & 45 & 49 & (7, 7) \\
\hline
15 & 58 & 63 & (7, 4, 4) \\
\hline
16 & 81 & 81 & (16), (4, 4, 4, 4) \\
\hline
17 & 95 & 112 & (10, 7) \\
\hline
18 & 127 & 147 & (7, 7, 4) \\
\hline
19 & 168 & 189 & (7, 4, 4, 4) \\
\hline
20 & 205 & 259 & (13, 7) \\
\hline
21 & 264 & 343 & (7, 7, 7) \\
\hline
22 & 340 & 441 & (7, 7, 4, 4) \\
\hline
23 & 413 & 592 & (13, 10) \\
\hline
24 & 523 & 784 & (10, 7, 7) \\
\hline
25 & 660 & 1029 & (7, 7, 7, 4) \\
\hline
26 & 806 & 1369 & (13, 13) \\
\hline
27 & 1002 & 1813 & (13, 7, 7) \\
\hline
28 & 1248 & 2401 & (7, 7, 7, 7) \\
\hline
29 & 1513 & 3087 & (7, 7, 7, 4, 4) \\
\hline
30 & 1866 & 4144 & (13, 10, 7) \\
\hline
31 & 2292 & 5488 & (10, 7, 7, 7) \\
\hline
32 & 2775 & 7203 & (7, 7, 7, 7, 4) \\
\hline
\end{tabular} 
\end{center}
\end{table}

Throughout this section, let $\lambda$ be a partition $(\lambda_1, \lambda_2, \ldots) \in P(n)$ such that $N(0, 3; \lambda)$ is maximal. First, we bound the size of $\lambda_1$.
\begin{prop}
\label{large03}
Assume the notation and hypotheses above. Then the following is true:
\[\lambda_1 \le 23.\]
\end{prop}
\begin{proof}
Suppose that $\lambda$ has a part $k \ge 24$. Then by Theorem~$\ref{ineq}$, replacing $k$ with the parts  $\lfloor{\frac{k}{2}}\rfloor$ and $\lceil{\frac{k}{2}}\rceil$ would yield a partition $\mu$ such that $N(0, 3; \mu) > N(0, 3; \lambda).$ This is a contradiction since $N(0,3;\lambda)$ is maximal.
\end{proof}
For $i>0$, let $m_i$ be the multiplicity of the part $i$ in $\lambda.$ We bound each $m_i$ for $i \neq 7$.
\begin{prop}
\label{mult03}
Assume the notation and hypotheses above. Then the following are true:
\[
\begin{cases}
m_i = 0 & i = 2, 5, 8, 11, 12, 14, 15, i \ge 17\\
m_i \le 1 & i = 3, 6, 9, 10, 16\\
m_i \le 3 & i = 1, 13\\
m_i \le 4 & i = 4.
\end{cases}
\] 
\end{prop}
\begin{proof}
If $i \ge 24$, then this follows from Proposition~$\ref{large03}$. If $i \le 23$ and $i \neq 1, 3, 4, 6, 7, 9, 10, 13, 16$, then replacing $i$ with the representation of $i$ in the Table~$\ref{03}$ would yield a partition $\mu$ with $N(0, 3; \mu) > N(0, 3; \lambda)$, so $m_i = 0$. For the remaining $i$, notice that the following replacements yield partitions $\mu$ with $N(0, 3; \mu) > N(0, 3; \lambda)$: 
\[(1, 1, 1, 1) \rightarrow (4); (3, 3) \rightarrow (6); (4, 4, 4, 4, 4) \rightarrow (13, 7); (6, 6) \rightarrow (4, 4, 4);  \]
\[(9, 9) \rightarrow (7, 7, 4); (10, 10) \rightarrow (13, 7); (13, 13, 13, 13) \rightarrow (10, 7, 7, 7, 7, 7, 7).\]
\end{proof}
Now, we present an improved bound for $m_i$ for $i=3, 6, 16$. 
\begin{prop}
\label{spcase03}
Assume the notation and hypotheses above. Then the following is true:
\[ m_3 = m_6 = m_{16} = 0 \]
unless $\lambda = (3)$, $(6)$, or $(16)$.
\end{prop}
\begin{proof}
If $m_a \ge 1$ for some $a$, by Proposition~$\ref{mult03}$, we know that $a = 3, 4, 6, 7, 9, 10, 13, 16$ and that $m_i \le 1$ for $i = 3, 6, 16$. 

Suppose that $m_3 = 1$ (resp. $m_6 = 1$, $m_{16} =1$). Then it can be verified that replacing $(a, 3)$ (resp. $(a, 6)$, $(a, 16)$) with the representation of $a+3$ (resp. $a+6$, $a+16$) in Table~$\ref{03}$ will produce a partition $\mu$ with $N(0, 3; \mu) > N(0, 3; \lambda).$ 
\end{proof}

We now impose restrictions on the pairs of distinct integers $a, b \neq 7$ that can simultaneously be present in $\lambda$.

\begin{prop}
\label{pairs03}
Assume the notation and hypotheses above. If $m_a = 1$ and $m_b = 1$ where $a > b$ and $a, b \neq 7$, then the following is true:
\[(a,b) = (4, 1), (13, 10). \]
\end{prop}
\begin{proof}
By Proposition~$\ref{mult03}$ and Proposition~$\ref{spcase03}$, we know that $a, b\in \left\{1, 4, 9, 10, 13 \right\}$. It can be verified that replacing $a$ and $b$ with the representation of $a+b$ in Table~$\ref{03}$ will yield a partition $\mu$ with $N(0, 3; \mu) > N(0, 3; \lambda)$ if $(a,b) \neq (4,1), (13, 10)$. 
\end{proof}

We now determine restrictions on the sets of natural numbers $a_1, a_2, \ldots, a_l \neq 7$ that can simultaneously be present in $\lambda$.
\begin{prop}
\label{triples03}
Assume the notation and hypotheses above. Suppose that $\lambda$ contains $a_1 \ge a_2 \ge \ldots \ge a_l$ such that $a_1, a_2, \ldots, a_l \neq 7$. Then $\lambda$ is one of the following: 
\[(a_1, a_2, \ldots a_l) = (1), (1, 1), (1, 1, 1), (3), (4), (4, 1), (4, 1, 1), (4,4), (4, 4, 4), (4, 4, 4, 4),\]\[ (6), (10), (13), (13, 10), (13, 13),  (13, 13, 10), (13, 13, 13), (16). \]
\end{prop}
\begin{proof}
By Proposition~$\ref{mult03}$, Proposition~$\ref{spcase03}$, and Proposition~$\ref{pairs03}$, we know that $(a_1, \ldots a_l)$ is either one of the above partitions, or it is one of the following (which we will rule out): 
\[(a_1, \ldots a_l) = (4, 1, 1, 1), (4, 4, 1), (4, 4, 1, 1), (4, 4, 1, 1, 1), (4, 4, 4, 1), (4, 4, 4, 1, 1), \]\[(4, 4, 4, 4, 1), (4, 4, 4, 4, 1, 1), (4, 4, 4, 4, 1, 1, 1), (13, 13, 13, 10). \] 

Let $a_t$ be $\sum_{j=1}^l a_j.$ Suppose that $(a_1, \ldots a_l) \neq (4, 1, 1), (4, 4, 4), (4, 4, 4, 4), (13, 13, 10), (13, 13, 13)$. If $a_t > 32$, then it can be verified that replacing $(a_1, \ldots a_l)$ with the representation of $a_t$ in Theorem~$\ref{bash3}$ will yield a partition $\mu$ with $N(0, 3; \mu) > N(0, 3; \lambda)$. If $a_t \le 32$, then replacing $(a_1, \ldots a_l)$ with the representation of $a_t$ in Table~$\ref{03}$ will yield a partition $\mu$ with $N(0, 3; \mu) > N(0, 3; \lambda)$. 
\end{proof}
Now, we will characterize the finitely many partitions $\lambda$ that contain a $1$ or a $4$.
\begin{prop}
\label{4103}
Assume the notation above. Suppose $m_1 \ge 1$ or $m_4 \ge 1$. Then $\lambda$ is one of the following partitions: 
\[(1), (1, 1), (1, 1, 1), (4), (4, 1), (4, 1, 1), (4, 4), (4, 4, 4), \]\[ (4, 4, 4, 4), (7, 4), (7, 7, 4), (7, 4, 4, 4),  (7, 7, 7, 7, 4). \]
\end{prop}
\begin{proof}
Suppose that $m_1 \ge 1$ or $m_4 \ge 1$. Consider the partition $\lambda_2$ obtained by deleting any parts of size $7$ from $\lambda$. Then by Proposition~$\ref{triples03}$, we know that  
\[\lambda_2 = (4, 1, 1), (4, 4, 4), (4, 4, 4, 4). \]
Now, we add back in the parts of size $7$. Notice that the following operations will produce a partition $\mu$ with $N(0, 3; \mu) > N(0, 3; \lambda)$: 
\[(7, 1) \rightarrow (4, 4); (7, 7, 4, 4, 4) \rightarrow (13, 13); \]\[ (7, 7, 7, 7, 4, 4) \rightarrow (13, 13, 10); (7, 7, 7, 7, 7, 4) \rightarrow (13, 13, 13).\] 
This proves the desired statement.
\end{proof}
We first consider the case where $n \le 32$ and prove the third and fourth columns of Table~$\ref{03}$.
\begin{proof}[Proof of Table~$\ref{03}$]
Consider the partition $\lambda_2$ obtained by deleting any parts with size $7$ from $\lambda$. From Proposition~$\ref{triples03}$ and Proposition~$\ref{4103}$, we can obtain all possible partitions $\lambda_2$. It can be verified that appending parts of size $7$ to these $\lambda_2$ yields exactly the partitions in Table~$\ref{03}$. It can be verified that in the case where multiple partitions $\lambda$ of $n$ remain, we have that the values $N(0, 3; \lambda)$ are all equal. The values of $\maxN(0, 3; n)$ can be deduced from this and the first two columns of Table~$\ref{03}$.
\end{proof}
We now suppose that $n \ge 33$. we further limit the sets of natural numbers $a_1, a_2, \ldots, a_l \neq 7$ that can simultaneously be present in $\lambda$.
\begin{prop}
\label{whee03}
Assume the notation and hypotheses above. For $n \ge 33$, suppose that $\lambda$ contains $a_1 \ge a_2 \ge \ldots \ge a_l$ such that $a_1, a_2, \ldots, a_l \neq 7$. Then $(a_1, a_2, \ldots a_l)$ is one of the following: 
\[(a_1, a_2, \ldots a_l) = (10), (13), (13, 10), (13, 13),(13, 13, 10), (13, 13, 13).\]
\end{prop}
\begin{proof}
By Proposition~$\ref{4103}$, we have that $m_1 = 0$ and $m_4 = 0$. Hence, the desired statement follows from Proposition~$\ref{triples03}$.
\end{proof}
\subsubsection{Proof of Theorem~$\ref{bash3}$ for $r=0$}
We now use Proposition~$\ref{whee03}$ to deduce Theorem~$\ref{bash3}$. Assume the notation in Section 3.1.1.
\begin{proof}[Proof of Theorem~$\ref{bash3}$ for $r=0$]
Consider the partition $\lambda_2$ obtained by deleting any parts with size $7$ from $\lambda$. Then by Proposition~$\ref{triples03}$, we know that  
\[\lambda_2 = (10), (13), (13, 10), (13, 13),(13, 13, 10), (13, 13, 13). \]
These partitions cover all the classes modulo $7$ of $n$ except for $n \equiv 0 \pmod 7$ exactly once. If $n \not\equiv 0 \pmod 7$, then appending parts of size $7$ to these partitions covers each $n$ exactly once and yields the partitions $\lambda$ in Theorem~$\ref{bash3}$. If $n \equiv 0 \pmod 7$, we can deduce that $\lambda = (7, 7, 7, \ldots, 7)$ as stated in Theorem~$\ref{bash3}$. The values for $\maxN(0, 3;n)$ can be deduced from this and the first two columns of Table~$\ref{03}$.
\end{proof}
\subsection{Proof of Theorem~$\ref{bash3}$ for $r=1$, $2$}
We prove Theorem~$\ref{bash3}$ for $r=1,2$ at the same time, since $N(1,3;n) = N(2,3;n)$. In Subsection 3.1.1, we study the combinatorial properties of $N(r, 3;\lambda)$ for $r=1, 2$ resulting from Theorem~$\ref{ineq}$ and the values of $N(r,3;n)$ for $r=1,2$ for small $n$. In Subsection 3.1.2, we use these properties to deduce Theorem~$\ref{bash3}$ for $r=1,2$.
\subsubsection{Some combinatorics for $r=1, 2$}
For simplicity of notation, we write our propositions in terms of $N(1,3;n)$ to denote the shared value $N(r,3;n)$ for $r=1,2$. In our combinatorial arguments, we require the values of $N(1,3;n)$ for $n\leq21$. These values are given in the first two columns of Table~$\ref{13}$, which were computed using Sage. We prove the correctness of the values in the last two columns over the course of this section. 
\begin{table}
\caption{}
\label{13}
\begin{center}
\begin{tabular}{|| c c c c ||}
\hline
$n$ & $N(1,3;n)$ & $\maxN(1,3;n)$ & $\lambda$ \\
\hline
\hline
1 & 0 & 0 & (1) \\
\hline
2 & 1 & 1 & (2) \\
\hline
3 & 1 & 1 & (3) \\
\hline
4 & 1 & 1 & (4), (2, 2) \\
\hline
5 & 3 & 3 & (5) \\
\hline
6 & 4 & 4 & (6) \\
\hline
7 & 4 & 4 & (7) \\
\hline
8 & 8 & 8 & (8) \\
\hline
9 & 10 & 10 & (9) \\
\hline
10 & 13 & 13 & (10) \\
\hline
11 & 20 & 20 & (11) \\
\hline
12 & 26 & 26 & (12) \\
\hline
13 & 32 & 32 & (13) \\
\hline
14 & 46 & 46 & (14) \\
\hline
15 & 59 & 59 & (15) \\
\hline
16 & 75 & 75 & (16) \\
\hline
17 & 101 & 101 & (17) \\
\hline
18 & 129 & 129 & (18) \\
\hline
19 & 161 & 161 & (19) \\
\hline
20 & 211 & 211 & (20) \\
\hline
21 & 264 & 264 & (21) \\
\hline
\end{tabular} 
\end{center}   
\end{table}

Let $\lambda$ be $(\lambda_1, \lambda_2, \ldots) \in P(n)$ be such that $N(1, 3; \lambda)$ is maximal. First, we bound the size of $\lambda_1$.
\begin{prop}
\label{large13}
Assume the notation and hypotheses above. Then the following is true:
\[\lambda_1 \le 21.\]
\end{prop}
\begin{proof}
Suppose that $\lambda$ has a part $k \ge 22$. Then by Theorem~$\ref{ineq}$, replacing $k$ with the parts  $\lfloor{\frac{k}{2}}\rfloor$ and $\lceil{\frac{k}{2}}\rceil$ would yield a partition $\mu$ such that $N(1, 3; \mu) > N(1, 3; \lambda).$ This is a contradiction since $N(1,3;\lambda)$ is maximal.
\end{proof}
For $i>0$, let $m_i$ be the multiplicity of the part $i$ in $\lambda.$ We bound each $m_i$ for $i \neq 14$.
\begin{prop}
\label{mult13}
Assume the notation and hypotheses above. Then the following are true:
\[
\begin{cases}
m_i = 0 & i \ge 22\\
m_i \le 1 & 1 \le i \le 22 \textit{ such that } i\neq 2, 11, 12, 14, 15 \\
m_i \le 2 & i = 2, 12, 15\\
m_i \le 3 & i = 11\\
\end{cases}
\] 
\end{prop}
\begin{proof}
For $i \ge 22$, this follows from Proposition~$\ref{large13}$. 

Suppose that $m_i \ge 2$ for $1 \le a \le 21$, $i \neq 2, 11, 12, 14, 15$. If $2i \ge 22$ (resp. $2i < 22$), it can be verified that replacing the parts $i$ and $i$ with the representation of $2i$ in Theorem~$\ref{bash3}$ (resp. Table~$\ref{13}$) would yield a partition $\mu$ with $N(1, 3; \mu) > N(1, 3; \lambda).$ 
For $i = 2, 11, 12, 15$, note that the following operations yield partitions $\mu$ with $N(1, 3; \mu) > N(1, 3; \lambda)$: 
\[(2, 2, 2) \rightarrow (6); (11, 11, 11, 11) \rightarrow (15, 15, 14); \]
\[(12, 12, 12) \rightarrow (14, 11, 11); (15, 15, 15) \rightarrow (12, 11, 11, 11).\]
\end{proof}
\begin{prop}
\label{spcase13}
Assume the notation and hypotheses above. If $\lambda$ contains a part of size $2$, then 
\[ \lambda = (2), (2, 2). \]
\end{prop}
\begin{proof}
Suppose that $\lambda$ contains a part of size $2$. For $i \neq 2$, if $i+2 \ge 22$ (resp. $< 22$), then replacing $i$ and $2$ with the representation of $i+2$ in Theorem~$\ref{bash3}$ (resp. Table~$\ref{13}$) will yield a partition $\mu$ with $N(1, 3; \mu) > N(1, 3; \lambda).$ This means $\lambda$ must contain only parts of size $2$. By Proposition~$\ref{mult13}$, we have that $m_2 \le 2$.
\end{proof}
We now impose restrictions on the pairs of distinct integers $a, b \neq 14$ that can simultaneously be present in $\lambda$.
\begin{prop}
\label{pairs13}
Assume the notation and hypotheses above. If $m_a = 1$ and $m_b = 1$ where $a > b$ and $a, b \neq 2, 14$, then the following is true:
\[(a,b) = (12, 11), (15, 12), (17, 15). \]
\end{prop}
\begin{proof}
By Proposition~$\ref{large13}$, we know that $a, b \le 21$.
If $(a,b) \neq (12, 11), (15, 12), (17, 15)$, it can be verified that replacing $a$ and $b$ with the representation of $a+b$ in Table~$\ref{03}$ will yield a partition $\mu$ with $N(1, 3; \mu) > N(1, 3; \lambda)$.
\end{proof}
We now impose restrictions on the sets of integers $a_1, a_2, \ldots, a_l \neq 14$ that can simultaneously be present in $\lambda$.
\begin{prop}
\label{triples13}
Assume the notation and hypotheses above. Suppose that $\lambda$ contains $a_1 \ge a_2 \ge \ldots \ge a_l$ such that $a_1, a_2, \ldots, a_l \neq 14$. Then $(a_1, a_2, \ldots a_l)$ is one of the following: 
\[(a_1, a_2, \ldots a_l) = (i) \textit{ (for } 1 \le i \le 21 \textit{)}, (2, 2), (11, 11), (11, 11, 11), \]\[ (12, 11), (12, 11, 11), (12, 12), (12, 12, 11), (15, 12), (15, 15), (17, 15). \]
\end{prop}
\begin{proof}
By Proposition~$\ref{mult13}$, Proposition~$\ref{spcase13}$, and Proposition~$\ref{pairs13}$, we know that $(a_1, a_2, \ldots, a_l)$ is either one of the above partitions or is one of the following (which we will rule out): 
 \[(a_1, a_2, \ldots, a_l) = (12, 11, 11, 11), (12, 12, 11, 11), (12, 12, 11, 11, 11), \]\[(15, 12, 12), (15, 15, 12), (15, 15, 12, 12), (17, 15, 15).\]
Let $a_t = \sum_{j=1}^l a_j$. 
Suppose that $(a_1, a_2, \ldots a_l)$ is not one of the sets in the statement. It can be verified that replacing $(a_1, a_2, \ldots a_l)$ with the representation of $a_t$ in Theorem~$\ref{bash3}$ will produce a partition will produce a partition $\mu$ with $N(1, 3; \mu) > N(1, 3; \lambda)$.
\end{proof}

We first consider the case where $n \le 21$ and prove the third and fourth columns of Table~$\ref{13}$.
\begin{proof}[Proof of Table~$\ref{13}$]
Consider the partition $\lambda_2$ obtained by deleting any parts with size $14$ from $\lambda$. From Proposition~$\ref{triples13}$, we can obtain all possible partitions $\lambda_2$. It can be verified that appending parts of size $14$ to these $\lambda_2$ yields exactly the partitions in Table~$\ref{13}$. It can be verified that in the case where multiple partitions $\lambda$ of $n$ remain, we have that the values $N(1, 3; \lambda)$ are all equal. The values of $\maxN(1, 3; n)$ can be deduced from this and the first two columns of Table~$\ref{13}$.
\end{proof}
We now suppose that $n \ge 22$. we further limit the sets of integers $a_1, a_2, \ldots, a_l \neq 14$ that can simultaneously be present in $\lambda$.
\begin{prop}
\label{whee13}
Assume the notation and hypotheses above. For $n \ge 22$, suppose that $\lambda$ contains $a_1 \ge a_2 \ge \ldots \ge a_l$ such that $a_1, a_2, \ldots, a_l \neq 14$. Then the following is true:
\[(a_1, a_2, \ldots a_l) =  (11), (11, 11), (11, 11, 11), (12), (12, 11), (12, 11, 11), (12, 12), \]\[ (12, 12, 11), (15), (15, 12), (15, 15), (17), (17, 15). \]
\end{prop}
\begin{proof}
The desired statement follows from Proposition~$\ref{spcase13}$ and Proposition~$\ref{triples13}$.
\end{proof}
\subsubsection{Proof of Theorem~$\ref{bash3}$ for $r= 1,2$}
We now use Proposition~$\ref{whee13}$ to deduce Theorem~$\ref{bash3}$ for $r=1,2$ using the fact that $N(1,3;n) = N(2,3;n)$. Assume the notation in Section 3.2.1. 
\begin{proof}[Proof of Theorem~$\ref{bash3}$ for $r=1,2$]
Since $N(1,3;n)$ is always equal to $N(2,3;n)$, we know that $\maxN(1,3;n)$ and $\maxN(2,3;n)$ are equal and are achieved at the same partitions. Consider the partition $\lambda_2$ obtained by deleting any parts with size $14
$ from $\lambda$. Then by Proposition~$\ref{triples03}$, we know that  
\[\lambda_2 = (11), (11, 11), (11, 11, 11),  (12), (12, 11), (12, 11, 11), \]\[(12, 12), (12, 12, 11), (15), (15, 12), (15, 15), (17), (17, 15).\]
These partitions cover all the classes modulo $14$ of $n$ except for $n \equiv 0 \pmod {14}$ exactly once. If $n \not\equiv 0 \pmod {14}$, then appending parts of size $14$ to these partitions covers each $n$ exactly once and yields the partitions $\lambda$ in Theorem~$\ref{bash3}$. If $n \equiv 0 \pmod {14}$, we can deduce that $\lambda = (14, 14, \ldots, 14)$ as stated in Theorem~$\ref{bash3}$. The values for $\maxN(1, 3;n)$ and $\maxN(2,3;n)$ can be deduced from this fact and the first two columns of Table~$\ref{13}$.
\end{proof}
\section{Discussion}
For general $t$, it is difficult to obtain effective asymptotics and effective bounds on error terms for $N(r,t;n)$. In particular, the exact formulas for $t=2$ as an infinite series were not known until a recent work by Bringmann and Ono \cite{bringono}. Using these bounds, we believe that similar methods can be used to prove the following convexity result.
\begin{conj}
If $t=2$ and $r=0$ (resp. $r=1$), then we have that
\[N(r, 2; a)N(r,2;b) > N(r,2;a+b) \]
for all $a,b \ge 11$ (resp. $12$).
\end{conj}
This convexity result would imply the following description of $\maxN(r,2;n)$: 
\begin{conj}
\label{bash2}
Assume the notation above. Then the following are true. 
\begin{enumerate}
\item{
If $n \ge 6$, then we have that 
\[ 
\maxN(0,2;n) = 
\begin{cases}
3^{\frac{n}{3}} & n \equiv 0 \pmod 3 \\
11 \cdot 3^{\frac{n-7}{3}} & n \equiv 1 \pmod 3 \\
5 \cdot 3^{\frac{n-5}{3}} & n \equiv 2 \pmod 3, \\
\end{cases} \]
and it is achieved at the unique partitions
\begin{align*}
(3, 3, \ldots, 3) &\textit{ when } n \equiv 0 \pmod 3 \\
(7, 3, \ldots, 3) &\textit{ when } n \equiv 1 \pmod 3  \\
(5, 3, \ldots, 3) &\textit{ when } n \equiv 2 \pmod 3.  \\
\end{align*}
}
\item{
If $n \ge 8$, then we have that  
\[ 
\maxN(1,2;n) = 
\begin{cases}
2^{\frac{n}{2}} & n \equiv 0 \pmod 2 \\
12 \cdot 2^{\frac{n-9}{2}} & n \equiv 1 \pmod 2, \\
\end{cases} \]
and it is achieved at the following classes of partitions
\begin{align*}
(2, 2, \ldots, 2) &\textit{ when } n \equiv 0 \pmod 2\\
(9, 2, \ldots, 2) &\textit{ when } n \equiv 1 \pmod 2. \\
\end{align*}
up to any number of the following substitutions: 
$(2, 2) \rightarrow (4)$ and $(2, 2, 2) \rightarrow (6).$} 
\end{enumerate}
\end{conj}
\begin{ex*}
For $n=8$, we would have that $\maxN(1,2;8) = 16$ is achieved at the partitions $(6,2)$, $(4,4)$, $(4,2,2)$, and $(2,2,2,2)$.
\end{ex*}
We believe that a similar convexity result holds for all $r,t$ for sufficiently large $a$ and $b$. 
\begin{conj}
If $0 \le r < t$ and $t \ge 2$, then
\[N(r, t; a)N(r,t;b) > N(r,t;a+b)\]
for sufficiently large $a$ and $b$. 
\end{conj}
\bibliography{biblio}
\bibliographystyle{plain}
\end{document}